\documentclass[reqno,a4paper]{amsart}
\usepackage{amssymb}
\usepackage{amsmath}
\newcommand{\half}{\frac{1}{2}}

\newcommand{\R}{\mathbb{R}}

\begin{document} 
\newtheorem{prop}{Proposition}[section]
\newtheorem{Def}{Definition}[section]
\newtheorem{theorem}{Theorem}[section]
\newtheorem{lemma}{Lemma}[section]
 \newtheorem{Cor}{Corollary}[section]

\title[Klein-Gordon-Schr\"odinger]{\bf Well-posedness results for a generalized Klein-Gordon-Schr\"odinger system}
\author[Hartmut Pecher]{
{\bf Hartmut Pecher}\\
Fakult\"at f\"ur  Mathematik und Naturwissenschaften\\
Bergische Universit\"at Wuppertal\\
Gau{\ss}str.  20\\
42119 Wuppertal\\
Germany\\
e-mail {\tt pecher@math.uni-wuppertal.de}}
\date{}

\begin{abstract}
We  consider the Klein-Gordon-Schr\"odinger system
$
i \partial_t \psi + \Delta \psi  = \phi^2 \psi - \phi \psi $ , $
(\Box +1)\phi = -2|\psi|^2 \phi + |\psi|^2$
with additional cubic terms and Cauchy data
$
\psi(0) = \psi_0 \in H^s(\R^n) \, , \, \phi(0) = \phi_0 \in H^k(\R^n) \, ,  \, (\partial_t \phi)(0) = \phi_1 \in H^{k-1}(\R^n)  $
in space dimensions $n=2$ and $n=3$ . We prove local existence, uniqueness and continuous dependence on the data in Bourgain-Klainerman-Machedon spaces  for low regularity data, e.g. for $s=-\frac{1}{8}$, $k=\frac{3}{8}+\epsilon$ in the case $n= 2$ and  $s=0$ , $k=\half+\epsilon$ in the case $n=3$. Global well-posedness in energy space is also obtained as  a special case. Moreover, we show "unconditional" uniqueness in the space
$\psi \in C^0([0,T],H^s)$, $\phi \in C^0([0,T],H^{s+\half}) \cap C^1([0,T],H^{s-\half})$, if $s > \frac{3}{22}$ for $n=2$ and $s > \half$ for $n=3$.
\end{abstract}
\maketitle
\renewcommand{\thefootnote}{\fnsymbol{footnote}}
\footnotetext{\hspace{-1.5em}{\it 2010 Mathematics Subject Classification:} 
35Q55, 35L70 }
\normalsize 
\setcounter{section}{0}

\section{Introduction}

\noindent 
Consider the Klein-Gordon-Schr\"odinger system
\begin{align}
\label{0.1}
i \partial_t \psi + \Delta \psi & = \phi^2 \psi - \phi \psi \\
\label{0.2}
(\Box +1)\phi & = -2|\psi|^2 \phi + |\psi|^2
\end{align}
with Cauchy data
\begin{equation} \label{0.3}
\psi(0) = \psi_0 \in H^s(\R^n) \, , \, \phi(0) = \phi_0 \in H^k(\R^n) \, ,  \, (\partial_t \phi)(0) = \phi_1 \in H^{k-1}(\R^n)  \end{equation}
in space dimensions $n=2$ and $n=3$ ,  where $\Box = \partial_t^2 - \Delta$ .

This system is a generalization of the classical Klein-Gordon-Schr\"odinger system (KGS) with Yukawa interaction of a complex nucleon field $\psi$ and a real meson field $\phi$. The classical KGS  system without the cubic nonlinearities was considered extensively in the literature. We are primarily interested in well-posedness results for data with low regularity.

For the classical KGS system in the case $n=2$ local well-posedness in Bourgain-Klainerman-Machedon type spaces (BKM spaces) holds for $s > - \frac{1}{4}$ , $k \ge - \half$, $k-2s <\frac{3}{2}$ , $k-2 < s < k+1$ . These conditions are sharp up to the endpoints. Uniqueness in the larger natural solution spaces
$ \psi \in C^0([0,T],H^s) $ , $ \phi \in C^0([0,T],H^k) \cap C^1([0,T],H^{k-1}) $  ("unconditional uniqueness") holds, if $s,k \ge 0$ . The solutions exist globally (in t), if $s \ge 0$ , $s-\half \le k < s +\half$. These results were shown by the author \cite{P1}.

In the case $n=3$ local well-posedness in BKM spaces holds under the same conditions in $s$ and $k$ except that we have to assume $ k > - \half$. This is also sharp up to the endpoints. The unconditional uniqueness result holds for $ s=k=0$ , in which case also global well-posedness is true. This was proven by the author in \cite{P}.

The proofs are based on the results by Bejenaru and Herr \cite{BH} in the case $n=3$ and Bejenaru-Herr-Holmer-Tataru \cite {BHHT} in the case $n=2$.

In the more general case (\ref{0.1}),(\ref{0.2}) with additional cubic interactions much less is known. For dimension $n=2$ Shi-Li-Wang \cite{SLW} proved global well-posedness in energy space $s=k=1$ using Kato type smoothing estimates for the nonlinear Schr\"odinger equation.

In the present paper this result can be improved by using BKM type spaces and the methods in the author's paper \cite{P1}. We obtain local well-posedness (existence, uniqueness and continuous dependence on the data) for $s > - \frac{1}{8}$ , $k >\frac{3}{8}$ and $s < \min(2k-\frac{1}{4},k+\half)$ , $ s > \max(\frac{k-1}{2},k-\frac{3}{2})$ , e.g. $s=-\frac{1}{8}+$ , $k=\frac{3}{8}+$ , in BKM spaces, which are subsets of the natural solution spaces $\psi \in C^0([0,T],H^s)$ , $\phi \in C^0([0,T],H^k) \cap C^1([0,T],H^{k-1})$ . "Unconditional" uniqueness in these latter spaces is shown for $s > \frac{3}{22}$ and (for simplicity) $k=s+\half$ . Global well-posedness in energy space easily follows as a byproduct. We rely mainly on the results in \cite{P1} combined with bilinear estimates for BKM spaces for the Klein-Gordon part by d'Ancona, Foschi and Selberg \cite{AFS}. We apply only an appropriate  version of the contraction mapping principle which also has the advantage that it is well-known that the solution depends continuously on the data. 

In the case $n=3$ Ran-Shi \cite{RS} proved two main results, first well-posedness in energy space $s=k=1$ by considering atomic solution  spaces $U^2$ and $V^2$ and the perturbed linear Schr\"odinger equation $i\partial_t + (\Delta + \phi)\psi =0$ . Secondly local well-posedness holds in the case $s > \half$ , $k=1$ , which improves earlier results by Shi-Wang-Li-Wang \cite{SWLW}.

In the present paper we use the standard contraction mapping principle and obtain local well-posedness in BKM type spaces for $s \ge 0$ , $k > \half$ , $s < \min(2k-\half,k+\half)$ , $ s > \max(\frac{k-1}{2},k-\frac{3}{2})$ , especially for $s=0$ , $k=\half +$. Unconditional uniqueness holds for $ s > \half$ , $k= s+\half$ . Global well-posedness in energy space also follows easily.

In fact, it is well-known that the system (\ref{0.1}),(\ref{0.2}) satisfies mass conservation
$$ \|\psi\|_{L^2} = \|\psi_0\|_{L^2} $$
and energy conservation
$$ E(t) \equiv \|\nabla \psi\|_{L^2}^2 + \half(\|\partial_t \phi\|_{L^2}^2 + \|\nabla \phi\|_{L^2}^2 + \|\phi\|_{L^2}^2) + 2\|\psi \phi\|_{L^2}^2 - \int \phi |\psi|^2 dx = E(0)\, . $$
One easily shows that, if  $\psi_0 \in H^1$ , $\phi_0 \in H^1$ , $ \phi_1 \in L^2$ , then $$E(0) \le c_0(\|\psi_0\|_{H^1},\|\phi_0\|_{H^1},\|\phi_1|_{L^2}) \, .$$ Moreover
\begin{align*}
|\int \phi |\psi|^2 dx | & \le \|\phi\|_{L^6} \||\psi|^{\half}\|_{L^{12}} \||\psi|^{\frac{3}{2}}\|_{L^{\frac{4}{3}}} \\
& \lesssim \|\nabla \phi\|_{L^2} \|\nabla \psi\|_{L^2}^{\half} \|\psi\|_{L^2}^{\frac{3}{2}} \\
& \le \frac{1}{4} (\| \nabla \phi\|_{L^2}^2 + \|\nabla \psi\|_{L^2}^2) + c \|\psi\|_{L^2}^6 \, ,
\end{align*}
so that
$$ \|\nabla \psi\|_{L^2}^2 + \half(\|\partial_t \phi\|_{L^2}^2 + \|\nabla \phi\|_{L^2}^2 + \|\phi\|_{L^2}^2) \le E(0) + \frac{1}{4} (\|\nabla \phi\|_{L^2}^2 + \|\nabla \psi\|_{L^2}^2) + c \|\psi_0\|_{L^2}^6 \, , $$
which implies the a-priori bound 
$$ \|\nabla \psi\|_{L^2}^2 + \|\partial_t \phi\|_{L^2}^2 + \|\nabla \phi\|_{L^2}^2 + \|\phi\|_{L^2}^2 \le c_1(\|\psi_0\|_{H^1},\|\phi_0\|_{H^1},\|\phi_1\|_{L^2}) \, . $$
Thus, if we have a unique local solution in $[0,T]$ for data $\psi_0 \in H^1$ , $\phi_0 \in H^1$ , $\phi_1 \in L^2$  with $T=T(\|\psi_0\|_{H^1},\|\phi_0\|_{H^1},\|\phi_1\|_{L^2}$ , this solution can easily be extended globally in time.

We only use the standard Bourgain spaces $X^{m,b}$ for the Schr\"odinger equation, which are defined as the completion of ${\mathcal S}(\R^n \times \R)$ with respect to the norm
$$ \|f\|_{X^{m,b}} := \| \langle \xi \rangle^m \langle \tau + |\xi|^2 \rangle^b \widehat{f}(\xi,\tau) \|_{L^2_{\xi\tau}} \, . $$
Similarly the Klainerman-Machedon spaces $X^{m.b}_{\pm}$ for the equation $$i \partial_t n_{\pm} \mp (-\Delta +1)^{\half} n_{\pm} =0$$ are the completion of ${\mathcal S}(\R^n \times \R)$ with respect to the norm
$$ \|f\|_{X_{\pm}^{m,b}} := \| \langle \xi \rangle^m \langle \tau \pm |\xi|\rangle^b \widehat{f}(\xi,\tau) \|_{L^2_{\xi\tau}} \, . $$
For a given time interval $I$ we define
$$ \|f\|_{X^{m,b}(I)} := \inf_ {\tilde{f}_{| I}=f} \|\tilde{f}\|_{X^{m,b}}$$
and similarly $\|f\|_{X^{m,b}_{\pm}(I)}$ .

We now formulate our main results.
\begin{theorem}
\label{Theorem}
(Local well-posedness) 
In the case $n=2$  assume
$$s \ge - \frac{1}{8} \, , \, k > \frac{3}{8} \, , \, s < \min(2k-\frac{1}{4},k+\half) \, , \, s > \max(\frac{k-1}{2},k-\frac{3}{2}) \, . $$
In the case $n=3$ assume
$$s \ge 0 \, , \, k > \half \, , \, s < \min(2k-\half,k+\half) \, , \, s > \max(\frac{k-1}{2},k-\frac{3}{2}) \, . $$
The Cauchy problem (\ref{0.1}),(\ref{0.2}),(\ref{0.3}) has a unique local solution
$$
 \psi \in X^{s,\half+}[0,T] \, , \, \phi \in X^{k,\half+}_+[0,T] + X^{k,\half+}_-[0,T] \, , $$
$$
 \partial_t \phi \in X^{k-1,\half+}_+[0,T] + X^{k-1,\half+}_-[0,T] \, , 
$$
where $T=T(\|\psi_0\|_{H^s},\|\phi_0\|_{H^k},\|\psi_1\|_{H^{k-1}})$ . This solution fufills
$$ \psi \in C^0([0,T],H^s) \, , \, \phi \in C^0([0,T],H^k) \cap C^1([0,T],H^{k-1}) \, . $$
\end{theorem}
{\bf Remark:} The solution depends continuously on the initial data and persistence of higher regularity holds. This is a standard fact for solutions obtained by a contraction argument as in our case.

\begin{theorem}
\label{Theorem1.2}
(Unconditional uniqueness) If we assume $k=s+\half$ , and $ s > \frac{3}{22}$ in the case $n=2$ and $s > \half$ in the case $n=3$ then there exists a unique solution of (\ref{0.1}),(\ref{0.2}),(\ref{0.3}) in the space $$ \psi \in C^0([0,T],H^s) \, , \, \phi \in C^0([0,T],H^k) \cap C^1([0,T],H^{k-1}) \, . $$
\end{theorem}

\section{Preliminaries}
We transform the KGS into an equivalent first order (in t) system as follows. Defining
$$ \phi_{\pm} = \phi \pm i(-\Delta +1)^{-\half} \partial_t \phi$$
we obtain the system
\begin{align*}
i \partial_t \psi + \Delta \psi & = (\phi_+ + \phi_-)^2 \psi - (\phi_++\phi_-)\psi \\
i \partial_t \phi_{\pm} \mp (-\Delta+1)^{\half} \phi_{\pm} & = \pm(-\Delta+1)^{-\half} (-2|\psi|^2(\phi_+ + \phi_-) + |\psi|^2)
\end{align*}
with data
$$ \psi(0)=\psi_0 \in H^s \, , \, \phi_{\pm}(0) = \phi_0 \pm i(-\Delta+1)^{-\half}\phi_1 \in H^k \, . $$

Fundamental for the proof of the multilinear estimates are also the following bilnear estimates in Klainerman-Machedon spaces spaces which were proven by d'Ancona, Foschi and Selberg in the cases $n=2$ in \cite{AFS} and $n=3$ in \cite{AFS1} in a more general form which include many limit cases which we do not need.
\begin{prop}
\label{AFS}
Let $n=2$ or $n=3$. The estimate
$$\|uv\|_{X_{\pm}^{-s_0,-b_0}} \lesssim \|u\|_{X^{s_1,b_1}_{\pm_1}} \|v\|_{X^{s_2,b_2}_{\pm_2}} $$
holds, where $\pm,\pm_1,\pm_2$ denote independent signs,
provided the following conditions are satisfied:
\begin{align*}
\nonumber
& b_0,b_1,b_2 \ge 0 \\
\nonumber
& b_0 + b_1 + b_2 > \frac{1}{2}  \\
\nonumber
&s_0+s_1+s_2 > \frac{n+1}{2} -(b_0+b_1+b_2) \\
\nonumber
&s_0+s_1+s_2 > \frac{n}{2} -(b_0+b_1) \\
\nonumber
&s_0+s_1+s_2 > \frac{n}{2} -(b_0+b_2) \\
\nonumber
 &s_0+s_1+s_2 > \frac{n}{2} -(b_1+b_2)\\
\nonumber
&s_0+s_1+s_2 > \frac{n-1}{2} - b_0 \\
\nonumber
&s_0+s_1+s_2 > \frac{n-1}{2} - b_1 \\
\nonumber
&s_0+s_1+s_2 > \frac{n-1}{2} - b_2 \\
\nonumber
&s_0+s_1+s_2 > \frac{n+1}{4} \\
\nonumber
&(s_0 + b_0) +2s_1 + 2s_2 > \frac{n}{2} \\
\nonumber
&2s_0+(s_1+b_1)+2s_2 > \frac{n}{2} \\
\nonumber
&2s_0+2s_1+(s_2+b_2) > \frac{n}{2} \\
\nonumber
&s_1 + s_2 \ge 0 \\
\nonumber
&s_0 + s_2 \ge 0 \\
\nonumber
&s_0 + s_1 \ge 0 \, .
\end{align*}
\end{prop}

\section{The case $n=2$}

We now prove the necessary bi- and trilinear estimates.
\begin{prop}
\label{Prop.1.1}
Let $ s > - \half$ , $k > \frac{3}{8}$ and $s< \min(2k-\frac{1}{4},k+\half)$ . Then the following estimate holds for sufficiently small $\epsilon > 0$ :
\begin{equation}
\label{1.0}
\|\phi_1 \phi_2 \psi\|_{X^{s,-\half+2\epsilon}} \lesssim  \|\phi_1\|_{X^{k,\half+\epsilon}_{\pm_1}} \|\phi_2\|_{X^{k,\half+\epsilon}_{\pm_2}} \|\psi\|_{X^{s,\half+\epsilon}} \, . 
\end{equation}
\end{prop}
We need the following lemmas for its proof.
\begin{lemma}
\label{Lemma1.1}
The following estimates hold for $ k > \frac{3}{8}$ :
\begin{equation}
\label{a}
\|\phi_1 \phi_2\|_{X^{2k-\frac{5}{4}-,\half-}_{\pm}} \lesssim \|\phi_1\|_{X^{k,\half+\epsilon}_{\pm_1}} \|\phi_2\|_{X^{k,\half+\epsilon}_{\pm_2}} \, ,
\end{equation}
if $k \le \frac{3}{4}$ , and
\begin{equation}
\label{b}
\|\phi_1 \phi_2\|_{X^{k-\half-,\half-}_{\pm}} \lesssim \|\phi_1\|_{X^{k,\half+\epsilon}_{\pm_1}} \|\phi_2\|_{X^{k,\half+\epsilon}_{\pm_2}} \, ,
\end{equation}
if $k > \frac{3}{4}$ .
\end{lemma}
\begin{proof}
Let $\xi_0 = \xi_1 + \xi_2$ , $\tau_0 = \tau_1 + \tau_2$ , where $\xi_j \in \R^n , \tau_j \in \R$ . The following elementary estimate holds:
\begin{align}
\nonumber
|\tau_0 \pm |\xi_0|| & =|\tau_1+\tau_2 \pm_1 |\xi_1| \pm_2 |\xi_2| \mp_1|\xi_1| \pm_2 |\xi_2| \pm |\xi_0|| \\
\label{1.1}
& \lesssim |\tau_1 \pm_1 |\xi_1|| + |\tau_2 \pm_2 |\xi_2|| + \max(|\xi_1|,|\xi_2|) \, .
\end{align}
By (\ref{1.1}) we  may reduce (\ref{a}) to the estimates
\begin{align}
\label{1.2}
\|\phi_1 \phi_2\|_{X^{2k-\frac{5}{4}-,0}_{\pm}} &\lesssim \|\phi_1\|_{X^{k-\half+,\half+\epsilon}_{\pm_1}} \|\phi_2\|_{X^{k,\half+\epsilon}_{\pm_2}} \, ,  \\
\label{1.3}
\|\phi_1 \phi_2\|_{X^{2k-\frac{5}{4}-,0}_{\pm}} &\lesssim \|\phi_1\|_{X^{k,0}_{\pm_1}} \|\phi_2\|_{X^{k,\half+\epsilon}_{\pm_2}} \, .
\end{align}

For (\ref{1.2}) we may use Prop. \ref{AFS} with parameters
$s_0 = \frac{5}{4}-2k+$ , $s_1=k-\half+$ , $s_2=k$ , $b_0=0$ , $b_1=b_2=\half+\epsilon$ , because  $s_0+s_1+s_2 = \frac{3}{4}+$ , $s_0+2s_1+2s_2 = \frac{3}{4}+2k-\half + > 1 $ for $k \ge \frac{3}{8}$ and $s_0+s_1= \frac{3}{4}-k+ > 0$ for $k \le \frac{3}{4}$ . 

For (\ref{1.3}) similarly we choose $s_0=\frac{5}{4}-2k+$ , $s_1=s_2=k$ , $b_0=b_1=0$ , $b_2=\half+\epsilon$ , so that $s_0+s_1+s_2 = \frac{5}{4}+$ , $s_1+s_2=2k \ge \frac{3}{4}$ , $s_0+s_2= \frac{5}{4}-k > 0$ . 

By (\ref{1.1}) we reduce (\ref{b}) to the estimates
\begin{align}
\label{1.4}
\|\phi_1 \phi_2\|_{X^{k-\half-,0}_{\pm}} &\lesssim \|\phi_1\|_{X^{k-\half+,\half+\epsilon}_{\pm_1}} \|\phi_2\|_{X^{k,\half+\epsilon}_{\pm_2}} \, ,  \\
\label{1.5'}
\|\phi_1 \phi_2\|_{X^{k-\half-,0}_{\pm}} &\lesssim \|\phi_1\|_{X^{k,0}_{\pm_1}} \|\phi_2\|_{X^{k,\half+\epsilon}_{\pm_2}} \, .
\end{align}

As above we use Prop. \ref{AFS}. Here we may choose for (\ref{1.4}) $s_0=\half-k+$ , $s_1=k-\half+$ , $s_2=k$ , $b_0=0$ , $b_1=b_2=\half+\epsilon$ , so that $s_0+s_1+s_2 =k+ > \frac{3}{4}$ , $s_1+s_2=2k-\half+ > 1$ and in the case (\ref{1.5'}) $s_0=\half-k+$ , $s_1=s_2=k$ , $b_0=b_1=0$, $b_2=\half+\epsilon$ , thus $s_0+s_1+s_2= k+\half+ > \frac{5}{4}$ .
\end{proof}

\begin{lemma}
\label{Lemma1.2}
For sufficiently small $\epsilon > 0$ the following estimates apply:
\begin{equation}
\label{a1}
\| |\phi|^2 \psi\|_{X^{s,-\half+2\epsilon}} \lesssim \||\phi|^2\|_{X^{2k-\frac{5}{4}-,\half-}_{\pm}} \|\psi\|_{X^{s,\half-}} \,  ,
\end{equation}
if $ k > \frac{3}{8}$ , $k \le \frac{3}{4}$ and $s > -\half$ , $s < 2k-\frac{1}{4}$ .
\begin{equation}
\label{b1}
\| |\phi|^2 \psi\|_{X^{s,-\half+2\epsilon}} \lesssim \||\phi|^2\|_{X^{k-\half-,\half-}_{\pm}} \|\psi\|_{X^{s,\half-}} \, ,
\end{equation}
if $ k >  \frac{3}{4}$ and $s > -\half$ , $s < k+\half$ .
\end{lemma}
\begin{proof}
This follows from \cite{P1}, Prop. 2.2. , where in the case (\ref{a1}) we have $\sigma = 2k-\frac{5}{4}-$, so that $\sigma \ge - \half$ for $k > \frac{3}{8}$ and $\sigma > s-1$ for $s < 2k-\frac{1}{4}$ . In the case (\ref{b1}) we choose $\sigma = k-\half-$ , so that $\sigma > s-1$ for $s < k+\half$ , and $\sigma > \frac{1}{4}$ .
\end{proof}

\begin{proof}[Proof of Prop. \ref{Prop.1.1}]
The proof follows immediately by Lemma \ref{Lemma1.1} and Lemma \ref{Lemma1.2}.
\end{proof}

\begin{prop}
\label{Prop.1.2}
Let $ s \ge - \frac{1}{8}$ , $ k > \frac{3}{8}$ and $s > \max(\frac{k-1}{2},k-\frac{3}{2})$ . The following estimate applies:
\begin{equation}
\label{1.4'}
\| |\psi|^2 \phi \|_{X^{k-1,-\half+2\epsilon}_{\pm}} \lesssim \|\psi\|^2_{X^{s,\half+\epsilon}} \|\phi\|_{X^{k,\half+\epsilon}_{\pm_2}} \, .
\end{equation}
\end{prop}

This is a consequence of the following two lemmas.
\begin{lemma}
\label{Lemma1.3}
The following estimates apply:
\begin{equation}
\label{a2}
\| |\psi|^2 \phi \|_{X^{k-1,-\half+2\epsilon}_{\pm}} \lesssim \| |\psi|^2\|_{X^{\frac{1}{4}-,-\half+2\epsilon}_{\pm_1}} \|\phi\|_{X^{k,\half+\epsilon}_{\pm_2}} \, ,
\end{equation}
if $\frac{3}{8} < k \le \frac{3}{4}$ , and
\begin{equation}
\label{b2}
\| |\psi|^2 \phi \|_{X^{k-1,-\half+2\epsilon}_{\pm}} \lesssim \| |\psi|^2\|_{X^{k-\half-,-\half+2\epsilon}_{\pm_1}} \|\phi\|_{X^{k,\half+\epsilon}_{\pm_2}} \, ,
\end{equation} 
if $k > \frac{3}{4}$ .
\end{lemma}
\begin{proof}
Estimate (\ref{a2}) is by duality equivalent to
$$ \|\phi w\|_{X^{-\frac{1}{4}+,\half-2\epsilon}_{\pm_1}} \lesssim \|\phi\|_{X^{k,\half+\epsilon}_{\pm_2}} \|w\|_{X^{1-k,\half-2\epsilon}_{\pm}} \, . $$
By (\ref{1.1}) this reduces to the following estimates:
\begin{align*}
\|\phi w\|_{X^{-\frac{1}{4}+,0}_{\pm_1}} &\lesssim \|\phi\|_{X^{k-\half+2\epsilon,\half+\epsilon}_{\pm_2}} \|w\|_{X^{1-k,\half-2\epsilon}_{\pm}} \, , \\
\|\phi w\|_{X^{-\frac{1}{4}+,0}_{\pm_1}} &\lesssim \|\phi\|_{X^{k,\half+\epsilon}_{\pm_2}} \|w\|_{X^{\half-k+2\epsilon,\half-2\epsilon}_{\pm}} \, , \\
\|\phi w\|_{X^{-\frac{1}{4}+,0}_{\pm_1}} &\lesssim \|\phi\|_{X^{k,3\epsilon}_{\pm_2}} \|w\|_{X^{1-k,\half-2\epsilon}_{\pm}} \, , \\
\|\phi w\|_{X^{-\frac{1}{4}+,0}_{\pm_1}} &\lesssim \|\phi\|_{X^{k,\half+\epsilon}_{\pm_2}} \|w\|_{X^{1-k,0}_{\pm}} \, .
\end{align*}
This follows by use of Prop. \ref{AFS}. One easily checks the necessary conditions. The assumption $k \le \frac{3}{4}$ is obviously necessary for the second estimate to hold.

Concerning (\ref{b2}) by similar arguments using (\ref{1.1}) we reduce to 4 estimates, e.g. we need
$$ \|\phi w\|_{X^{\half-k+,0}_{\pm_1}} \lesssim \|\phi\|_{X^{k-\half+2\epsilon,\half + \epsilon}_{\pm_2}} \|w\|_{X^{1-k,\half-2\epsilon}_{\pm}} \, , $$
which by Prop. \ref{AFS} with parameters $s_0=k-\half-$ , $s_1=k-\half+2\epsilon$ , $s_2=1-k$ , $b_0=0$ , $b_1=\half+\epsilon$ , $b_2=\half-2\epsilon$ requires $s_0+s_1+s_2 \ge k > \frac{3}{4}$ . Similarly the other estimates can be proven.
\end{proof}

\begin{lemma}
\label{Lemma1.4'}
The following estimates apply:
\begin{equation}
\label{a3}
\| |\psi|^2\|_{X^{\frac{1}{4}-,-\half+}_{\pm}} \lesssim \|\psi\|^2_{X^{s,\half-}} 
\end{equation}
for $s \ge - \frac{1}{8}$ , and
\begin{equation}
\label{b3}
\| |\psi|^2\|_{X^{k-\half-,-\half+}_{\pm}} \lesssim \|\psi\|^2_{X^{s,\half-}} 
\end{equation}
for $ s >\max(\frac{k-1}{2},k-\frac{3}{2})$ and $k > \frac{3}{4}$ .
\end{lemma}
\begin{proof}
We apply \cite{P1}, Prop. 2.4 in the case (\ref{a3}) with parameter $\sigma = \frac{5}{4}-$ , so that $\sigma < 2s + \frac{3}{2}$ for $s \ge - \frac{1}{8}$ and $\sigma < s+2$ , and in the case (\ref{b3}) with $\sigma = k+\half-$ , thus $\sigma < 2s+\frac{3}{2}$ for $s > \frac{k-1}{2}$ and $\sigma < s+2$ for $s > k-\frac{3}{2}$ .
\end{proof}

The estimates for the quadratic nonlinearities are as follows.
\begin{prop}
\label{Prop.1.3}
The following estimates hold:
\begin{equation}
\label{1.5}
\|\phi \psi\|_{X^{s,-\half+2\epsilon}} \lesssim \|\phi\|_{X^{k,\half+\epsilon}_{\pm}} \|\psi\|_{X^{s,\half+\epsilon}} 
\end{equation}
for $ s > - \half$ , $k \ge - \half$ , $ s < k+1$ , and
\begin{equation}
\label{1.6}
\| |\psi|^2\|_{X^{k-1,-\half+2\epsilon}_{\pm}} \lesssim \|\psi\|^2_{X^{s,\half+\epsilon}}
\end{equation}
for $s > - \frac{1}{4}$ , $ s>k-2$ , $s > \frac{k}{2}-\frac{3}{4}$ .
\end{prop}
\begin{proof}
\cite{P1}, Prop. 2.2 and Prop. 2.4.
\end{proof}

\begin{proof}[Proof of Theorem \ref{Theorem}]
By well-known arguments these propositions imply the \\
claimed local well-posedness result in Theorem \ref{Theorem} in the case $n=2$ .
\end{proof}

Our next aim is to prove the unconditional uniqueness result Theorem \ref{Theorem1.2}.
We prepare this by the following lemma.
\begin{lemma}
\label{Lemma1.4}
Let $|k| < \frac{1}{6}$. The following estimate applies:
$$ \|\psi \phi_{\pm}\|_{X^{k,-\frac{5}{12}-}} \lesssim \|\psi\|_{X^{k,\frac{5}{12}+}} \|\phi_{\pm}\|_{X^{-\half+,\frac{5}{12}+}_{\pm}} \, . $$
\end{lemma}
\begin{proof}
This is a variant of \cite{P1}, Prop. 2.1 and a straightforward consequence of its proof. Let $\xi_1$ and $\xi_2$ denote the spatial frequencies of $\widehat{\psi \phi_{\pm}}$ and $\widehat{\psi}$, respectively. If $|\xi_1| \sim |\xi_2|$ the estimate is in immediate consequence of this proposition for any $k \in \R$ . If $|\xi_1| \ll |\xi_2|$ or $|\xi_1| \gg |\xi_2|$ we may use case 2 of its proof (which follows from \cite{BH}, Prop. 4.8) and requires $|k| < \frac{1}{6}$ .
\end{proof}

\begin{proof}[Proof of Theorem \ref{Theorem1.2}] 
We assume $\psi_0 \in H^s(\R^2)$ , $\phi_0\in H^k(\R^2)$ , $\phi_1 \in H^{k-1}(\R^2)$, where $k=s+\half$ for the sake of simplicity. We want to show that any solution $\psi \in C^0([0,T],H^s)$ , $\phi \in C^0([0,T],H^{s+\half}) \cap C^1([0,T],H^{s-\half})$ for sufficiently large $s$ fulfills $\psi \in X^{-\frac{1}{8}+,\half+}[0,T]$ , $\phi \in X^{\frac{3}{8}+,\half+}_+[0,T] + X^{\frac{3}{8}+,\half+}_-[0,T]$ , where uniqueness holds by Theorem \ref{Theorem}.

A first step is the following lemma.
\begin{lemma}
\label{Lemma1.5}
If $ \frac{1}{6} < s \le \frac {1}{3}$ the following regularity holds:
\begin{align*}
\psi &\in X^{3s-1,1}[0,T] \cap X^{s,0}[0,T] \subset X^{\frac{11}{6}s-\frac{5}{12}-,\frac{5}{12}+}[0,T] \cap X^{2s-\half-,\half+}[0,T] \, ,\\
\phi_{\pm} &\in X^{3s-\half,1}_{\pm}[0,T] \cap X^{s+\half,0}_{\pm}[0,T] \subset X^{\frac{11}{6}s+
\frac{1}{12}-,\frac{5}{12}+}_{\pm}[0,T] \cap X^{2s-,\half+}_{\pm}[0,T] \, . 
\end{align*}
\end{lemma}
\begin{proof}
We use an idea of Zhou \cite{Z} and estimate by Sobolev:
$$ \|\phi^2 \psi\|_{L^2([0,T],H^l)} \lesssim \|\phi^2 \psi\|_{L^2([0,T],L^r)} \lesssim T^{\half} \|\phi\|^2_{L^{\infty}([0,T],H^{s+\half})} \|\psi\|_{L^{\infty}([0,T],H^s)} < \infty \, , $$
where $\frac{1}{r} = \half - \frac{l}{2}$ , so that $L^r \hookrightarrow H^l$, and $\frac{1}{p} = \half - \frac{s+\half}{2}$ , $\frac{1}{q}=\half-\frac{s}{2}$ , so that $H^{s+\half} \hookrightarrow L^p$ and $H^s \hookrightarrow L^q$ . By H\"older we need $\frac{1}{r}=\frac{2}{p}+\frac{1}{q}$, which holds for $l=3s-1$ , so that $l \le 0$ for $s\le \frac{1}{3}$ and $\frac{1}{r} = 1-\frac{3}{2}s < 1$ for $s > 0$ .  Similarly we obtain
\begin{align*}
 \||\psi|^2 \phi\|_{L^2([0,T],H^m)} &\lesssim \||\psi|^2 \phi\|_{L^2([0,T],L^{r^*})} \\ &\lesssim T^{\half} \|\psi\|^2_{L^{\infty}([0,T],H^{s})} \|\phi\|_{L^{\infty}([0,T],H^{s+\half})} < \infty \, , 
\end{align*}
where $\frac{1}{r^*} = \half - \frac{m}{2}$ , $p$ and $q$ as before. By H\"older we need $\frac{1}{r^*}=\frac{1}{p}+\frac{2}{q}$, which holds for $m=3s-\frac{3}{2}$ , so that $m\le 0$ for $s \le \half$ and $\frac{1}{r^*} = \frac{5}{4}-\frac{3}{2}s < 1$ for $s >\frac{1}{6}$ .  The quadratic nonlinearities are handled in the same way. Because $ s \le \half$ we also have $\psi_0 \in H^{3s-1}$ , $\phi_0 \in H^{3s-\half}$ and $\phi_1 \in H^{3s-\frac{3}{2}}$ , so that by standard arguments by (\ref{0.1})  and (\ref{0.2}) we obtain $\psi \in X^{3s-1,1}[0,T]$ and $\phi_{\pm} \in X^{3s-\half,1}[0,T]$ . The rest follows by interpolation.
\end{proof}

This lemma immediately implies $ \psi \in X^{-\frac{1}{8}+,\half+}[0,T]$ and $\phi_{\pm} \in X^{\frac{3}{8}+,\half+}[0,T]$ for $ s > \frac{3}{16} $ , so that Theorem \ref{Theorem} gives unconditional uniqueness in this case.

Now we improve this lower bound on $s$ .
\begin{lemma}
\label{Lemma1.6}
The following estimate applies:
$$ \|\phi_1 \phi_2\|_{X^{-\half+,\frac{5}{12}+}_{\pm}} \lesssim \|\phi_1\|_{X^{\frac{1}{3}+\epsilon,\frac{5}{12}+}} \|\phi_2\|_{X^{\frac{1}{3}+\epsilon,\frac{5}{12}+}} \, . $$
\end{lemma}
\begin{proof}
By (\ref{1.1}) we have to prove
\begin{align}
\label{a5}
\|\phi_1 \phi_2\|_{X^{-\half+,0}_{\pm}} &\lesssim \|\phi_1\|_{X^{-\frac{1}{12}+\frac{\epsilon}{2},\frac{5}{12}+}} \|\phi_2\|_{X^{\frac{1}{3}+\epsilon,\frac{5}{12}+}} \, , \\
\label{b5}
\|\phi_1 \phi_2\|_{X^{-\half+,0}_{\pm}} &\lesssim \|\phi_1\|_{X^{\frac{1}{3}+\epsilon,0}} \|\phi_2\|_{X^{\frac{1}{3}+\epsilon,\frac{5}{12}+}} \, .
\end{align}
(\ref{a5}) follows by Prop. \ref{AFS} with parameters $s_0=\half -$ , $s_1=-\frac{1}{12}+\frac{\epsilon}{2}$ , $s_2=\frac{1}{3}+\epsilon$, $b_0=0$ , $b_1=b_2=\frac{5}{12}+$ , thus $s_0+s_1+s_2 = \frac{3}{4}+$ and $s_1+s_2 = \frac{1}{4}+$ . Similarly (\ref{b5}) follows with $s_0=\half-$ , $s_1=s_2=\frac{1}{3}+\epsilon$ , $b_0=b_1=0$ , $b_2=\frac{5}{12}+$ , thus $s_0+s_1+s_2 = \frac{7}{6}+$.
\end{proof}

Now we assume $ \frac{1}{3} \ge s > \frac{3}{22}$ . We want to improve the regularity of $\psi$ . By Lemma \ref{Lemma1.5} we know $\psi \in X^{-\frac{1}{6}+\epsilon,\frac{5}{12}+}$ and $\phi_{\pm} \in X^{\frac{1}{3}+\epsilon,\frac{5}{12}+_{\pm}}$ . By Lemma \ref{Lemma1.4} and Lemma \ref{Lemma1.6} we obtain
\begin{align*}
\|\phi_1 \phi_2 \psi\|_{X^{-\frac{1}{6}+\epsilon,-\frac{5}{12}+}} & \lesssim  \|\phi_1 \phi_2\|_{X^{-\half+,\frac{5}{12}+}_{\pm}} \|\psi\|_{X^{-\frac{1}{6}+\epsilon,\frac{5}{12}+}} \\
& \lesssim \|\phi_1\|_{X^{\frac{1}{3}+\frac{5}{12}+}_{\pm_1}} \|\phi_2\|_{X^{\frac{1}{3}+\frac{5}{12}+}_{\pm_2}}  \|\psi\|_{X^{-\frac{1}{6}+\epsilon,\frac{5}{12}+}} \, .
\end{align*}
The quadratic term is easily treated in the same way. This implies
$$\psi \in  X^{-\frac{1}{6}+\epsilon,\frac{7}{12}-}[0,T] \cap X^{\frac{3}{22}+,0}[0,T] \subset X^{-\frac{19}{154},\half+}[0,T] \subset X^{-\frac{1}{8}+\epsilon,\half+}[0,T] \, . $$
It remains to prove iteratively $\phi_{\pm} \in X^{\frac{3}{8}+,\half+}_{\pm}[0,T]$ , starting by Lemma \ref{Lemma1.5} with $\phi_{\pm} \in X^{2s-,\half +}_{\pm} \subset X^{\frac{3}{11}+,\half+}_{\pm} \subset X^{\frac{1}{4}+,\half+}_{\pm}$ . We want to prove $\phi_{\pm} \in X^{k_{j+1},\half+}_{\pm} $ , if $\phi_{\pm} \in X^{k_j,\half+}_{\pm}$ with $k_{j+1} > k_j$ and $k_0 \ge \frac{1}{4}$ . We obtain
\begin{align*}
\| |\psi|^2 \phi_{\pm}\|_{X^{k_{j+1}-1,-\half+}_{\pm}} & \lesssim \| |\psi|^2 \|_{X^{\frac{1}{4}+2\epsilon,-\half +}_{\pm_1}} \|\phi_{\pm}\|_{X^{k_j,\half+}_{\pm_2}} \\
& \lesssim \|\psi\|^2_{X^{-\frac{1}{8}+\epsilon+,\half-}} \|\phi_{\pm}\|_{X^{k_j,\half+}_{\pm_2}} \, .
\end{align*}
For the last estimate we used \cite{P}, Prop. 2.4 with parameter $\sigma = \frac{5}{4}+2\epsilon$ , and the first estimate is by duality equivalent to
$$ \|\phi_{\pm} w\|_{X^{-\frac{1}{4}-2\epsilon,\half-}_{\pm_1}} \lesssim \|\phi_{\pm}\|_{X^{k_j,\half+}_{\pm_2}} \|w\|_{X^{1-k_{j+1},\half-}_{\pm}} \, . $$
By (\ref{1.1}) this reduces to the following estimates:
\begin{align}
\label{a6}
\|\phi_{\pm} w\|_{X^{-\frac{1}{4}-2\epsilon,0}_{\pm_1}} & \lesssim \|\phi_{\pm}\|_{X^{k_j - \half+,\half+}_{\pm_2}} \|w\|_{X^{1-k_{j+1},\half-}_{\pm}} \\
\label{b6}
\|\phi_{\pm} w\|_{X^{-\frac{1}{4}-2\epsilon,0}_{\pm_1}} & \lesssim \|\phi_{\pm}\|_{X^{k_j,\half+}_{\pm_2}} \|w\|_{X^{\half-k_{j+1},\half-}_{\pm}} \\
\label{c6}
\|\phi_{\pm} w\|_{X^{-\frac{1}{4}-2\epsilon,0}_{\pm_1}} & \lesssim \|\phi_{\pm}\|_{X^{k_j,0+}_{\pm_2}} \|w\|_{X^{1-k_{j+1},\half-}_{\pm}} \\
\label{d6}
\|\phi_{\pm} w\|_{X^{-\frac{1}{4}-2\epsilon,0}_{\pm_1}} & \lesssim \|\phi_{\pm}\|_{X^{k_j,\half+}_{\pm_2}} \|w\|_{X^{1-k_{j+1},0}_{\pm}} \, .
 \end{align}
We apply Prop. \ref{AFS}. For (\ref{a6}) we choose $s_0=\frac{1}{4}+2\epsilon$ , $s_1=k_j -\half+$ , $s_2=1-k_{j+1}$, $b_0 =0$ , $b_1=\half+$ , $b_2=\half-$ , thus $s_0+s_1+s_2= \frac{3}{4}+ 2\epsilon + k_j - k_{j+1} + > \frac{3}{4}$ , if $k_{j+1} = k_j + 2\epsilon$ . Moreover $s_1+s_2 = \half + k_j-k_{j+1}+ = \half-2\epsilon+$ , so that the assumptions of Prop. \ref{AFS} are satisfied. The other estimates can be handled in the same way, as one easily checks. The quadratic term is also easily treated, because we obtain as before
$$ \| |\psi|^2\|_{X^{\frac{1}{4}+2\epsilon,-\half+}_{\pm}} \lesssim \|\psi\|^2_{X^{-\frac{1}{8}+\epsilon+,\half+}} \, . $$

After finitely many steps we obtain $\phi_{\pm} \in X^{\frac{3}{8}+,\half+}_{\pm}[0,T]$ , so that we arrive in a space, where uniqueness holds by Theorem \ref{Theorem}.
\end{proof}

{\bf Remarks:} 1. It is also possible to generalize Theorem \ref{Theorem1.2} to data $\psi_0 \in H^s$ , $\phi_0 \in H^k$ , $\phi_1 \in H^{k-1}$ for $k \neq s+\half$ under appropriate conditions on $s$ and $k$. \\
2. If we consider the finite energy case $s=k=1$ , where by energy conservation and Theorem \ref{Theorem} global well-posedness holds for $\psi \in X^{1,\half+}$ , $\phi_{\pm} \in X^{1,\half+}_{\pm}$ , we can easily show unconditional uniqueness. One only has to estimate by Sobolev
$$ \| \phi^2 \psi\|_{L^2([0,T],L^2)} \lesssim T^{\half} \|\phi\|^2_{L^{\infty}([0,T],H^1)} \|\psi\|_{L^{\infty}([0,T],H^1)} < \infty $$
and also
$$ \| |\psi|^2 \phi\|_{L^2([0,T],L^2)} < \infty \, ,$$
 so that
$ \psi \in X^{0,1}[0,T]  \cap X^{1,0}[0,T] \subset X^{\half-,\half+}[0,T]$ and $ \phi_{\pm} \in X^{1,1}_{\pm}[0,T] $ , where uniqueness holds by Theorem {\ref{Theorem}. This holds true in dimensions $n=2$ as well as $n=3$.

\section{The case $n=3$}
We start with the proof of the local well-posedness result.

\begin{prop}
\label{Prop.2.1}
Let $s > - \half$ , $k > \half$ and $s < \min(2k-\half,k+\half)$ . Then the estimate (\ref{1.0}) applies.
\end{prop}
This is a consequence of the following two lemmas.
\begin{lemma}
\label{Lemma2.1}
The following estimates apply for $ k > \half$ :
\begin{equation}
\label{a8}
\|\phi_1 \phi_2\|_{X^{2k-\frac{3}{2}-,\half-}_{\pm} } \lesssim \|\phi_1\|_{X^{k,\half+\epsilon}_{\pm_1}}  \|\phi_2\|_{X^{k,\half+\epsilon}_{\pm_2}} \, ,
\end{equation}
if $ k \le 1$ , and
\begin{equation}
\label{b8}
\|\phi_1 \phi_2\|_{X^{k-\half-,\half-}_{\pm}} \lesssim \|\phi_1\|_{X^{k,\half+\epsilon}_{\pm_1}}  \|\phi_2\|_{X^{k,\half+\epsilon}_{\pm_2}} \, ,
\end{equation}
if $ k > 1$ .
\end{lemma}
\begin{proof}
 By (\ref{1.1}) we may reduce (\ref{a8}) to
$$
\|\phi_1 \phi_2\|_{X^{2k-\frac{3}{2}-,0}_{\pm}}  \lesssim \|\phi_1\|_{X^{k-\half-,\half+\epsilon}_{\pm_1}}  \|\phi_2\|_{X^{k,\half+\epsilon}_{\pm_2}} $$
and
$$\|\phi_1 \phi_2\|_{X^{2k-\frac{3}{2}-,0}_{\pm}}  \lesssim \|\phi_1\|_{X^{k,0}_{\pm_1}}  \|\phi_2\|_{X^{k,\half+\epsilon}_{\pm_2}} \, .$$
These estimates follow from Prop. \ref{AFS}. Similarly we handle (\ref{b8}) by reduction  to
$$
\|\phi_1 \phi_2\|_{X^{k-\half-,0}_{\pm}}  \lesssim \|\phi_1\|_{X^{k-\half+,\half+\epsilon}_{\pm_1}}  \|\phi_2\|_{X^{k,\half+\epsilon}_{\pm_2}} $$
and
$$\|\phi_1 \phi_2\|_{X^{k-\half-,0}_{\pm}}  \lesssim \|\phi_1\|_{X^{k,0}_{\pm_1}} \|\phi_2\|_{X^{k,\half+\epsilon}_{\pm_2}} \, .$$ 
\end{proof}

\begin{lemma}
\label{Lemma2.2}
The following estimates apply:
\begin{equation}
\label{a9}
\| |\phi|^2  \psi\|_{X^{s,-\half+2\epsilon}} \lesssim \| |\phi|^2 \|_{X^{2k-\frac{3}{2}-,\half-}_{\pm}} \|\psi\|_{X^{s,\half-}} 
\end{equation}
for $k >\half$ , $k \le 1$ , $s<2k-\half$ , $s >-\half$ , and
\begin{equation}
\label{b9}
\| |\phi|^2 \psi\|_{X^{s,-\half+2\epsilon}} \lesssim \| |\phi|^2 \|_{X^{k-\half-,\half-}_{\pm}} \|\psi\|_{X^{s,\half-}} 
\end{equation}
for $k > 1$ , $s < k+\half$ , $s > - \half$ .
\end{lemma}
\begin{proof}
This is a consequence of \cite{P}, Thm. 2.1, where in case (\ref{a9}) we choose $\sigma = 2k-\frac{3}{2}- >\half$ , thus $\sigma > s-1$ for $s < 2k-\half$ , $\sigma > -s-1 $ for $k > \half$ and $s > -\half$, whereas in case (\ref{b9}) we have $\sigma = k-\half- > \half$ , $\sigma > s-1$ for $s < k+\half$ , thus $s > -\sigma-1 = -k+\half$ for $ k > 1$ and $s >-\half$ .
\end{proof}

Next we have to prove
\begin{prop}
\label{Prop.2.2}
If $ s \ge 0$ , $k > \half$ and $s > \max(\frac{k-1}{2},k-\frac{3}{2})$, the estimate (\ref{1.4'}) applies.
\end{prop}
This is a consequence of the following two lemmas.
\begin{lemma}
\label{Lemma2.3}
If $\half < k < 1$ we have
$$ \| |\psi|^2 \phi\|_{X^{k-1,-\half+2\epsilon}_{\pm}} \lesssim \| |\psi|^2 \|_{X^{\half-,-\half+2\epsilon}_{\pm_1}} \|\phi\|_{X^{k,\half+\epsilon}_{\pm_2}} $$
and, if $ k \ge 1$ we have
$$ \| |\psi|^2 \phi\|_{X^{k-1,-\half+2\epsilon}_{\pm}} \lesssim \| |\psi|^2 \|_{X^{k-\half+,-\half+2\epsilon}_{\pm_1}} \|\phi\|_{X^{k,\half+\epsilon}_{\pm_2}} \, .$$
\end{lemma}
\begin{proof}
The first estimate is by duality equivalent to
$$ \| \phi w\|_{X^{-\half+,\half-2\epsilon}_{\pm_1}} \lesssim \| \phi \|_{X^{k,\half + \epsilon}_{\pm_2}} \|w\|_{X^{1-k,\half-2\epsilon}_{\pm}} \, . $$
This is handled similarly as in Lemma \ref{Lemma1.3} by (\ref{1.1}) and Prop. \ref{AFS}. The second estimate is handled in the same way.
\end{proof}
\begin{lemma}
\label{Lemma2.4}
If $s \ge 0$ the following estimate applies:
$$ \| |\psi|^2 \|_{X^{\half-,-\half+}_{\pm}} \lesssim \| \psi\|^2_{X^{s,\half+}} \, , $$
and, if $s > \max(k-\frac{3}{2},\frac{k-1}{2})$ and $ k \ge 1$ :
$$\| |\psi|^2 \|_{X^{k-\half+,-\half+}_{\pm}} \lesssim \| \psi\|^2_{X^{s,\half+}} \, . $$
\end{lemma}
\begin{proof}
This follows from \cite{P}, Thm. 2.2, where for the first estimate we choose $\sigma = \frac{3}{2}-$ , so that $2s > \sigma - \frac{3}{2}$ , if $s \ge 0$ , and for the second estimate $\sigma = k+\half+$ , thus $2s> \sigma-\frac{3}{2}=k-1$ , if $s > \frac{k-1}{2}$ , and $s > \sigma -2 = k-\frac{3}{2}+$ .
\end{proof}
The estimates for the quadratic nonlinearities are given in the next proposition.
\begin{prop}
\label{Prop.2.3}
The estimate (\ref{1.5}) applies for $s >-1$ , $k> -\half$ , $k>-1-s$ , $s<k+1$ , and the estimate (\ref{1.6}) holds for  $s> \max(\frac{k}{2}-\frac{3}{4},k-2)$ , $s>-\frac{1}{4}$ .
\end{prop}
\begin{proof}
This was proven in \cite{P}, Thm. 2.1 and Thm. 2.2.
\end{proof}
\begin{proof}[Proof of Theorem \ref{Theorem}]
By standard arguments the claimed local well-posedness in Theorem \ref{Theorem} is a consequence of Propositions \ref{Prop.2.1}, \ref{Prop.2.2} and \ref{Prop.2.3}.
\end{proof}

Next, we prove the unconditional well-posedness result.
\begin{proof}[Proof of Theorem \ref{Theorem1.2}]
As in the two-dimensional case we obtain
\begin{lemma}
\label{Lemma2.5}
Let $ \frac{1}{3} < s \le \frac{2}{3}$ . Then we have
\begin{align*}
\psi &\in X^{3s-2,1}[0,T] \cap X^{s,0}[0,T] \subset X^{2s-1-,\half+}[0,T] \, , \\
 \phi_{\pm} &\in X^{3s-\frac{3}{2},1}_{\pm}[0,T] \cap X^{s+\half,0}_{\pm}[0,T] \subset X^{2s-\half-,\half+}_{\pm}[0,T] \, . 
\end{align*}
\end{lemma}
\begin{proof}
We use the same notation as in the proof of Lemma \ref{Lemma1.5}, i.e. $\frac{1}{r}=\half-\frac{l}{2}$, $\frac{1}{p}=\half-\frac{s+\half}{3}$ , $\frac{1}{q} = \half-\frac{s}{3}$ , so that $\frac{1}{r}= \frac{2}{p} + \frac{1}{q}= \frac{7}{6}-s < 1$ , if $s > \frac{1}{6}$ and $l=3s-2 \le 0$ , if $s \le \frac{2}{3}$ . Moreover $\frac{1}{r^*} = \half - \frac{m}{3}$ , $\frac{1}{r^*} = \frac{1}{p} + \frac{2}{q} = \frac{4}{3}-s < 1$ , if $s > \frac{1}{3}$ , and $m=3s-\frac{5}{2} \le 0$ for $s \le \frac{5}{6}$ , This gives the necessary estimates for the cubic nonlinearities exactly as in Lemma \ref{Lemma1.5}. The quadratic nonlinearities are easily treated in the same way.
\end{proof}
Finally, in order to obtain the claimed unconditional uniqueness of Theorem \ref{Theorem1.2} it suffices to consider the case $s=\half+\epsilon$ . By Lemma \ref{Lemma2.5} we obtain $\psi \in X^{0,\half+}[0,T]$ and $\phi \in X^{\half+,\half+}[0,T]$ . But in these spaces the uniqueness holds by Theorem \ref{Theorem}.
\end{proof}

\end{document}